\documentclass[runningheads]{llncs}
\usepackage[T1]{fontenc}
\usepackage{graphicx}
\usepackage{xcolor}
\usepackage{amsmath,amssymb}
\usepackage{algorithm}
\usepackage{algorithmic}
\usepackage{booktabs}
\usepackage{mathtools}

\newtheorem{assumption}{Assumption}

\newcommand{\R}{\mathbb{R}}

\newcommand{\inner}[2]{\langle #1, #2\rangle}
\newcommand{\norm}[1]{\left\lVert #1\right\rVert}
\newcommand{\xs}{x^{\star}}
\newcommand{\Fh}{\widehat{F}}
\newcommand{\hloc}{\hat{L}_{\mathrm{loc}}}
\DeclareMathOperator*{\argmin}{arg\,min}

\title{Robust Decentralized Optimization under Node Failures via Adaptive Regularization\thanks{The research was supported by Russian Science Foundation (project No. 21-71-30005-$\Pi$), https://rscf.ru/en/project/21-71-30005/.}}
\author{Ilya Kuruzov\inst{1,2}\orcidID{0000-0002-2715-5489} \and Anastasiia Murzina \inst{3,4}\orcidID{0009-0009-8488-5769}
}
\authorrunning{I. Kuruzov et al.}
\institute{Moscow Institute of Physics and Technology, Moscow 117303, Russia\\ 
\and Innopolis University, Innopolis 420500, Russia\\
\and Sirius University of Science and Technology,  Sirius Federal Territory 354340, Russia\\
\and ITMO University,  St. Petersburg 197101, Russia \\
\email{kuruzov.ia@phystech.edu, murzinanastasiia@gmail.com}
}

\begin{document}

\maketitle              

\begin{abstract}
We study decentralized minimization of a sum of functions over a network where nodes may only leave, the remaining nodes stay connected, and the topology freezes between departures. Standard methods forget departed functions, causing a permanent bias proportional to the heterogeneity of the data. We propose Legacy Gradient Tracking (Legacy‑GT): before leaving, a node compresses its function into a gradient‑anchored quadratic legacy, bequeaths it to a neighbor, and provides a correction that exactly preserves the gradient‑tracking invariant. We prove that the optimal legacy curvature is the average of the strong‑convexity and smoothness constants, that legacies compose losslessly across chained departures, and that an adaptive anchor rule yields error bounds that decay geometrically after the last departure to a small residual—the minimum of the network’s optimization error at departure and its heterogeneity radius. In contrast, the classic drop‑and‑forget baseline suffers a bias that never decays. Numerical experiments confirm the theory.

\end{abstract}

 \\

\section*{Introduction}

Consider $n$ agents that cooperate over a communication network to solve
\begin{equation}\label{eq:problem}
\min_{x\in\R^d}\; F(x)\;=\;\sum_{i=1}^{n} f_i(x),
\end{equation}
where the function $f_i$ (a local loss over private data) is available only at node $i$. Decentralized first-order methods solve \eqref{eq:problem} without any coordinator by interleaving local gradient steps with gossip averaging. Decentralized subgradient descent converges only to a neighborhood of the optimum under a constant stepsize \cite{nedic2009}; exact linear rates for smooth, strongly convex objectives were later recovered by EXTRA \cite{shi2015extra} and by gradient-tracking schemes \cite{dimarco2016next,nedic2017diging,qu2018harnessing}, which extend to time-varying and directed graphs via push-sum \cite{nedic2015pushsum,nedic2017diging}, with stochastic variants and unified analyses given by \cite{pu2021distributed} and \cite{koloskova2020unified}. All of these methods, however, assume a \emph{fixed node set}: only the edges evolve over time, and convergence is always to the minimizer of the sum over that fixed set.

We focus on a ubiquitous form of time variation that these methods do not cover, \emph{monotone departure}, in which the node set itself shrinks and the optimization target drifts unless memory of departed nodes is explicitly preserved: \textbf{(N1)} $V^{t+1}\subseteq V^{t}\subseteq[n]$: nodes can only leave; \textbf{(N2)} every $G^t=(V^t,E^t)$ is connected, and $V^{t+1}=V^t\Rightarrow E^{t+1}=E^t$: the topology changes only when membership does. This captures sensors dying out, preempted cloud workers, cross-silo federations losing participants, and volunteer networks churning --- in every case \emph{after} the departing node has already contributed data-dependent computation.

\textbf{The forgetting problem.}
Every classical decentralized method restarted or continued on the survivors $S=V^{\infty}$ converges to $\xs_S=\argmin_x\sum_{i\in S}f_i(x)$: departed functions leave no trace in the fixed point. The resulting bias,
\begin{equation}\label{eq:dropbias}
\norm{\xs_S-\xs}\;=\;\Theta\!\Big(\tfrac{k\,\zeta}{n\,\mu}\Big),
\qquad \zeta=\max_i\norm{\nabla f_i(\xs)},
\end{equation}
is \emph{constant}: it does not shrink however late the $k$ nodes leave and however long the survivors iterate afterwards. Yet a node that leaves after the network has reached accuracy $\varepsilon$ has almost fully expressed its influence: its iterate is $\varepsilon$-close to $\xs$ and its local gradient there encodes exactly the pull it would have kept exerting. Discarding this information is wasteful, and --- as we show --- unnecessary. We next situate this gap relative to the adjacent literatures before describing our method and contributions.

A closely related line of work studies \emph{open} multi-agent systems, in which agents arrive and depart during execution: gossip and consensus under churn are analyzed in \cite{hendrickx2017open,franceschelli2021stability}, with fundamental performance limits derived by \cite{degalland2022fundamental}. Two papers on open-network \emph{optimization} are closest to ours. \cite{hendrickx2020stability} characterizes the sensitivity of the global minimizer to the removal of a single function, with tight bounds of order $O\big(\min\{\kappa^{1/2},\,\kappa/n^{1/2},\,\kappa^{3/2}/n\}\big)$ times the heterogeneity radius; we adopt this bound as a tighter drop-branch certificate inside our anchor rule \eqref{eq:anchor} (Remark~\ref{rem:hr}). \cite{hsieh2021open} proposes decentralized dual averaging that tracks the time-varying optimizer of the currently present sum --- optimally \emph{forgetting} absent nodes --- whereas we recover the \emph{original} full-sum minimizer $\argmin\sum_{i\in[n]}f_i$. Neither work transmits a function summary at departure, the mechanism that lets our bias decay with the departure round rather than settle at a heterogeneity-dependent floor.

In federated learning, FedAvg-style analyses handle \emph{recurring} partial participation, where every client returns often enough to keep shaping the limit \cite{mcmahan2017fedavg,li2020fedprox,karimireddy2020scaffold,yang2022anarchic}; permanent dropout is qualitatively different and reproduces the bias \eqref{eq:dropbias}. Two federated mechanisms anticipate pieces of our construction: the proximal anchor of FedProx \cite{li2020fedprox} (a quadratic penalty toward a reference point, used there for stability and here for memory) and the control variates of SCAFFOLD \cite{karimireddy2020scaffold} (stored stale gradients correcting for absent participants). Relatedly, SAG and SAGA \cite{schmidt2017sag,defazio2014saga} converge linearly using stale gradients evaluated at old iterates, showing that a first-order snapshot retains most of a function's optimization value near the optimum; our legacy is precisely such a frozen snapshot, augmented with a curvature term optimally tuned at $\eta_i=\frac{\mu_i+L_i}{2}$.

Personalized federated learning couples a local model to a reference point through the same quadratic penalty we use: pFedMe optimizes Moreau envelopes \cite{dinh2020pfedme}, L2GD mixes local and global models through an identical coupling \cite{hanzely2020federated}, and Ditto trades fairness against a global reference with the same term \cite{li2021ditto}. There, however, the penalty lets a \emph{present} client deviate from consensus; we use the identical object in reverse, to make an \emph{absent} client's pull persist, with a closed-form optimal weight rather than a tuned one (Lemma~\ref{lem:center}). Elastic Weight Consolidation \cite{kirkpatrick2017ewc} similarly protects past tasks with a Fisher-weighted quadratic penalty around previous parameters; our legacies instantiate the same principle in a fully analyzable convex setting, with the weight derived optimally rather than heuristically. Finally, machine unlearning \cite{cao2015towards,bourtoule2021machine} pursues the opposite goal of erasing a departed participant's influence; our bound on how much influence a first-order snapshot can preserve (Theorem~\ref{thm:lower}) doubles as a bound on how much such snapshots leak, of possible independent interest for unlearning audits. Taken together, no prior decentralized method preserves the influence of permanently departing nodes on the exact limit point with a quantified, departure-time-dependent bias, even though each ingredient we combine --- gradient tracking (GT), proximal anchors, stale gradients, curvature-weighted penalties --- is individually classical.

\textbf{Our approach: legacies instead of amnesia.}
Rather than treating a departure as erasure, we let the departing node compress its function into a \emph{legacy}: a quadratic surrogate
$q_i(x)=f_i(y_i)+\inner{\nabla f_i(y_i)}{x-y_i}+\frac{\eta_i}{2}\norm{x-y_i}^2$
anchored at a carefully chosen point $y_i$, transmitted in a single farewell message and adopted by one surviving neighbor. Four design choices make this simple idea provably work.
\emph{(i)~First-order anchoring:} the penalty carries the local \emph{gradient} at the anchor, not just the anchor point; without it the memory cannot encode the direction of $f_i$'s residual pull, and the achievable bias degrades by a factor of order $n$ (Sec.~\ref{sec:eta}).
\emph{(ii)~Optimal frozen curvature:} the weight $\eta_i=\frac{\mu_i+L_i}{2}$ is computed \emph{once}, at the departure round; we prove this midpoint is optimal and that a weight growing after departure would pin the limit to the anchor and destroy convergence.
\emph{(iii)~Adaptive anchor:} each node additionally runs a communication-free local descent on its own $f_i$ toward the local minimizer $x_i^{\ast}=\argmin f_i$; at departure, the anchor is chosen between the consensus iterate (best for \emph{late} exits, when it is near $\xs$) and this local track (best for \emph{early} exits, when it is the only meaningful summary of $f_i$).
\emph{(iv)~One heir, exact hand-off:} the legacy is inherited by a single neighbor together with a one-line tracker correction that preserves the gradient-tracking invariant exactly, so the running objective stays an ordinary decentralized sum and no restart is needed.
Table~\ref{tab:compare} previews the resulting error compensation.

\begin{table}
\caption{\textbf{All guarantees shown are for the strongly convex regime (SC).} Limit‑point error after $k$ departures for strongly convex objectives.
  Methods: classic drop (bias never decays), GT‑P (point‑only penalty),
  GT‑M (local‑minimizer anchor), and Legacy‑GT (adaptive anchor).
  $\varepsilon_m$, $\hat\varepsilon^{\mathrm{loc}}_m$ denote the network and
  local‑descent errors at departure $t_m$; $\hat\zeta_{i_m}$ is the
  certified heterogeneity.  Message size in $d$‑vectors; anchored variants
  carry one extra tracker. }
\label{tab:compare}
\centering
\begin{tabular}{|l|l|l|l|}
\hline
Method & No. of  msgs & Limit-point error & guar. \\
\hline
classic (drop) & --- & $\Theta\big(\frac{k\zeta}{n\mu}\big)$, \emph{never decays} & Eq.~\eqref{eq:dropbias} \\
GT-P & $1$v & $\lesssim \sum_m\min\{2\varepsilon_m,\frac{\hat\zeta_{i_m}}{(n-k)\mu}\}$ & Prop.~\ref{prop:pointonly} \\
GT-M & $2$v & $\le\frac{\kappa-1}{2n}\sum_m(\frac{\zeta}{\mu}+\hat\varepsilon^{\mathrm{loc}}_m)$ & Thm.~\ref{thm:main} \\
\textbf{Legacy-GT} & $3$v & $\le\frac{\kappa-1}{n}\sum_m\min\{\varepsilon_m,\frac{\zeta}{\mu}+\hat\varepsilon^{\mathrm{loc}}_m\}$ & Thm.~\ref{thm:main}, Cor.~\ref{cor:selective} \\
\hline
\end{tabular}
\end{table}

\textbf{Our Contribution.}\label{sec:contrib}
On the method side (Sec.~\ref{sec:algo}), Legacy-GT augments GT with a departure protocol in which the leaving node sends one farewell message to a single heir, and a closed-form tracker correction preserves the tracking invariant exactly (Prop.~\ref{prop:invariant}). The curvature $\eta_i^\star=\frac{\mu_i+L_i}{2}$ is proven optimal and composes losslessly across chained departures (Sec.~\ref{sec:eta}, Lemma~\ref{lem:center}, Prop.~\ref{prop:compose}), and pairs with an adaptive anchor --- consensus iterate or local track, whichever is certifiably closer --- for robustness to early exits.

On the theory side, Theorem~\ref{thm:main} (Sec.~\ref{sec:theory}) bounds the limit-point error, for strongly convex objectives, by the classic linear-rate term plus a legacy bias with prefactor $\frac{\kappa-1}{n}$ decaying geometrically in the departure times, unlike the non-decaying classic bias \eqref{eq:dropbias}. A point-only fallback (Sec.~\ref{sec:meta}) trails full gradient anchoring by $\Theta(n/(\kappa-1))$, and Theorem~\ref{thm:lower} shows the guarantee is near-optimal, within a factor of $\kappa$ of the worst-case lower bound $\Omega\big(\frac{(L-\mu)\varepsilon}{(n-1)\mu+L}\big)$ for any first-order-snapshot memory.

\section{Problem Statement and Assumptions}\label{sec:setup}
We use $\norm{\cdot}$ for the Euclidean norm, $[n]=\{1,\dots,n\}$, and $\bar x^t$ for the average iterate of the current survivors.

\begin{assumption}[General curvature class]\label{as:fun}
Each $f_i:\R^d\to\R$ is $b_i$-smooth with curvature bounded below by $a_i\le b_i$, i.e., $f_i-\frac{a_i}{2}\norm{\cdot}^2$ is convex and $f_i-\frac{b_i}{2}\norm{\cdot}^2$ is concave. We write $\omega_i:=\frac{b_i-a_i}{2}$ (the \emph{curvature spread}) and distinguish two regimes:
\begin{itemize}
\item[\textbf{(SC)}] \emph{strongly convex}: $a_i=\mu_i>0$, $b_i=L_i$; then $\mu=\min_i\mu_i$, $L=\max_i L_i$, $\kappa=L/\mu$; $\xs=\argmin F$ and $x_i^{\ast}=\argmin f_i$ are unique; heterogeneity $\norm{\nabla f_i(\xs)}\le\zeta$ (so $\norm{x_i^{\ast}-\xs}\le\zeta/\mu_i$); $R=\max_i\norm{x_i^0-\xs}$;
\item[\textbf{(C)}] \emph{convex}: $a_i=0$, $b_i=L_i$; the solution set $X^{\star}=\argmin F$ is nonempty but possibly unbounded, and only function values are meaningful.
\end{itemize}
\end{assumption}
Further, we define $l_{\max}=\max_i a_i.$
\begin{assumption}[Certified anchor errors]\label{as:anchors}
At its departure round $t_m$, node $i_m$ holds anchor candidates with \emph{certified} error bounds: $\varepsilon^{\mathrm{glob}}_m$ bounding the distance of its consensus iterate to the base method's current target, and (whenever $x_{i_m}^{\ast}$ exists) $\varepsilon^{\mathrm{loc}}_m$ bounding $\norm{u_{i_m}^{t_m}-x_{i_m}^{\ast}}$. 
\end{assumption}
These are abstract inputs to the general analysis; each regime of Assumption~\ref{as:fun} instantiates them concretely: geometrically decaying via the linear rate in (SC) (Proposition~\ref{fact:gt}), from the $O(1/t)$ value rate combined with quadratic growth in (C) (Proposition~\ref{prop:certC}).

\begin{assumption}[Network and departures]\label{as:net}
(N1)--(N2) hold. Departures occur at rounds $t_1<\dots<t_k$; node $i_m$ leaves at the end of round $t_m$, leaving survivors $S_m$, $|S_m|=n-m$, $S_0=[n]$; at least one node never leaves. Each epoch graph admits a doubly stochastic mixing matrix (e.g., Metropolis weights, computable from local degrees) with spectral gap at least $1-\lambda>0$.
\end{assumption}

Simultaneous departures execute the protocol independently toward surviving heirs (one exists while the graph is connected); the analysis orders them arbitrarily. \textbf{Abrupt crashes:} each node checkpoints $(y,g,\eta,z)$ to a buddy every $\tau$ rounds; a crash costs at most $\tau$ rounds of anchor staleness, a $(1-\theta)^{-\tau/2}$ constant for $\tau=O(1/\theta)$. \textbf{Load balancing:} heirs are chosen with fewest legacies, keeping $\hloc\le L(1+\ell_{\max})$ and the rate loss bounded. \textbf{Privacy:} the farewell message reveals one function value, one gradient, and one point --- strictly less than what unlearning audits assume adversaries already extract from trained iterates \cite{bourtoule2021machine}; if even this is too much, the point-only variant \eqref{eq:etat} applies.

\begin{assumption}[Graceful exit]\label{as:grace}
A departing node can send one final message to one neighbor of its choice. 
\end{assumption}

\section{The General Algorithm}\label{sec:algo}

Our goal is a protocol, executable within Assumptions~\ref{as:net}--\ref{as:grace}, whose limit point is provably close to the minimizer $\xs$ of the \emph{original} problem \eqref{eq:problem}, with an error that is dominated by \eqref{eq:dropbias} and \emph{vanishes as departures happen later}. The baseline is the \emph{classic} algorithm: GT run on the survivors with departed nodes removed and trackers restarted; it converges linearly to $\xs_{S_k}$, and the bias \eqref{eq:dropbias} is tight.

\subsection{Legacy surrogates and adaptive anchoring}
When node $i$ departs at round $t$, it freezes an anchor $y_i$, the gradient $g_i=\nabla h_i(y_i)$ of its working function $h_i$ ($f_i$ plus any inherited legacies), and a curvature $\eta_i$, forming the {legacy}
\begin{equation}\label{eq:legacy}
q_i(x)\;=\;h_i(y_i)+\inner{g_i}{x-y_i}+\tfrac{\eta_i}{2}\norm{x-y_i}^2 .
\end{equation}
The gradient term is essential: a pure quadratic ($g_i=0$) has zero gradient at the anchor while $f_i$ generally does not, so its bias converges to the heterogeneity $\zeta_i=\norm{\nabla f_i(\xs)}$ even as $y_i\to\xs$; with the gradient term, $\norm{\nabla q_i(x)-\nabla f_i(x)}\le\frac{L_i-\mu_i}{2}\norm{x-y_i}$ (Lemma~\ref{lem:center}) vanishes with anchor accuracy. The one case where a pure quadratic is (almost) exact is $y_i=x_i^\ast$, where $\nabla f_i(x_i^\ast)=0$ --- this defines the variant GT-M below. The legacy is handed to a single neighbor (the \emph{heir}), so after $m$ departures the running objective remains an ordinary decentralized sum,
\begin{equation}\label{eq:surrogate}
\Fh_m(x)=\sum_{j\in S_m} h_j(x)=\sum_{j\in S_m} f_j(x)+\sum_{\ell=1}^{m} q_{i_\ell}(x),\qquad x^{\star}_m=\argmin \Fh_m,\ \ \Fh_0=F,\ x^{\star}_0=\xs,
\end{equation}
and since $\eta_i\ge\mu_i$, every original index contributes at least $\mu$ of curvature, so $\Fh_m$ is $n\mu$-strongly convex for every $m$ (Lemma~\ref{lem:mu}).

Which point should the legacy remember? Each node tracks two candidate sequences: the consensus iterate $x_i^{t}\to x^{\star}_m\approx\xs$, with certified error $\hat\varepsilon(t)=C\sqrt{n}\,\bar R\,(1-\theta)^{t/2}$ (Proposition~\ref{fact:gt}), and the communication-free local track $u_i^{t+1}=u_i^{t}-\frac{1}{L_i}\nabla f_i(u_i^{t})\to x_i^{\ast}$, with certified error $\hat\varepsilon_i^{\mathrm{loc}}(t)=(1-\kappa_i^{-1})^{t}\norm{u_i^0-x_i^{\ast}}$. At the departure round $t_m$ --- and only then --- the node compares three certified bounds,
\begin{align*}
b^{\mathrm{cons}}_i&=\hat\varepsilon(t_m), &
b^{\mathrm{min}}_i&=\zeta/\mu_i+\hat\varepsilon^{\mathrm{loc}}_i(t_m), &
b^{\mathrm{drop}}_i&=\tfrac{1}{2\beta}\cdot\tfrac{\hat\zeta_i}{(n{-}m)\mu},
\end{align*}
where $\hat\zeta_i=\norm{\nabla f_i(x_i^{t_m})}+L_i\hat\varepsilon(t_m)\ge\zeta_i$ certifies the classic drop bias \eqref{eq:dropbias}, rescaled by $\tfrac{1}{2\beta}$ for comparability with the anchored bounds (Theorem~\ref{thm:main}), and picks
\begin{equation}\label{eq:anchor}
y_i=\begin{cases}
x_i^{t_m} & b^{\mathrm{cons}}_i \text{ is smallest},\\[2pt]
u_i^{t_m} & b^{\mathrm{min}}_i \text{ is smallest},\\[2pt]
\varnothing & b^{\mathrm{drop}}_i \text{ is smallest (no legacy)}.
\end{cases}
\end{equation}
Once chosen, $y_i$ is frozen forever, like the weight (Remark~\ref{rem:frozen}): a node leaving \emph{late} remembers where the collective was heading ($b^{\mathrm{cons}}$ decays geometrically and eventually wins); a node leaving \emph{early} remembers where it was heading itself, off by at most the heterogeneity radius.

\begin{remark}[When no anchor is best]\label{rem:noanchor}
Since $2\beta\,b^{\mathrm{min}}_i/b^{\mathrm{drop}}_i\approx\frac{\kappa_i-1}{2}$, a node with an ill-conditioned local function ($\kappa_i>3$) leaving early is better off contributing no legacy: an inaccurate anchor plus a large curvature mismatch injects more bias than forgetting. Rule \eqref{eq:anchor} detects this from observables and falls back to the classic bias term (Cor.~\ref{cor:selective}); late departures never trigger this branch, since $b^{\mathrm{cons}}\to0$ while $b^{\mathrm{drop}}$ stays bounded away from zero.
\end{remark}

\begin{remark}[Sharper drop certificates]\label{rem:hr}
The open-systems sensitivity bound of \cite{hendrickx2020stability} --- $O\big(\min\{\kappa^{1/2},\kappa/n^{1/2},\kappa^{3/2}/n\}\big)$ times the spread of local minimizers --- is a drop-in, strictly tighter replacement for $b^{\mathrm{drop}}_i$ in \eqref{eq:anchor} whenever $\kappa\gg1$, with no change to the protocol.
\end{remark}

\subsection{Legacy Gradient Tracking and its variants}
Algorithm~\ref{alg:lgt} gives the full method. The base method is adapt-then-combine GT: each survivor holds an iterate $x_j$ and a tracker $z_j$ estimating the average gradient of the current objective \eqref{eq:surrogate}. All variants differ only in the farewell message: \textbf{GT-C} (consensus anchor) sends $y_i=x_i^{t_m}$ with weight $\eta_i=\frac{\mu_i+L_i}{2}+\eta_i^{\mathrm{led}}$, and its bias decays geometrically in $t_m$; \textbf{GT-M} (minimizer anchor) sends $y_i=u_i^{t_m}$ with the same weight --- essentially a pure proximity penalty since $\nabla f_i(u_i^{t_m})\approx0$ --- and its bias $\le2\beta k(\zeta/\mu+\hat\varepsilon^{\mathrm{loc}}(t_1))$ is independent of how early the exit is (both by Thm.~\ref{thm:main}); \textbf{Legacy-GT} (adaptive) selects among GT-C, GT-M, and no legacy per departure by rule \eqref{eq:anchor}, so each node contributes the minimum of the three certified bounds (Thm.~\ref{thm:main}, Cor.~\ref{cor:selective}); \textbf{GT-P} (point-only; fallback for bandwidth/privacy-constrained exits) sends only $y_i=x_i^{t_m}$ and a frozen scalar weight \eqref{eq:etat} capped at $\eta^{(t_m)}_i\le\frac{c_0}{\alpha}$ (Prop.~\ref{prop:pointonly}).

\begin{figure*}[!t]
\begin{minipage}{\textwidth}
\hrule\vspace{3pt}
\caption{Legacy-GT (code for node $j$; round $t$)}
\label{alg:lgt}
\vspace{2pt}\hrule\vspace{3pt}
\begin{algorithmic}[1]
\REQUIRE stepsize $\alpha=\Theta\big(\tfrac{(1-\lambda)^2}{\hloc}\big)$; $h_j\!=\!f_j$; $z_j^0\!=\!\nabla h_j(x_j^0)$; local track $u_j^0=x_j^0$; curvature ledger $\eta_j^{\mathrm{led}}=0$.
\STATE \textit{// --- standard gradient-tracking round on current graph $G^t$ ---}
\STATE $x_j^{t+1}=\sum_{l\in N_j\cup\{j\}} w_{jl}^{t}\,(x_l^{t}-\alpha z_l^{t})$
        \COMMENT{mix a local descent step; $w^t$: Metropolis weights of $G^t$}
\STATE $z_j^{t+1}=\sum_{l} w_{jl}^{t} z_l^{t}+\nabla h_j(x_j^{t+1})-\nabla h_j(x_j^{t})$
        \COMMENT{tracker: preserves $\sum_j z_j=\sum_j\nabla h_j(x_j)$}
\STATE $u_j^{t+1}=u_j^{t}-\tfrac{1}{L_j}\nabla f_j(u_j^{t})$
        \COMMENT{auxiliary local run toward $x_j^{\ast}$; no communication}
\STATE \textit{// --- departure protocol: node $j$ leaves at end of round $t$ ---}
\IF{node $j$ departs now}
\STATE choose anchor $y_j\in\{x_j^{t},\,u_j^{t},\,\varnothing\}$ by rule \eqref{eq:anchor}, evaluated once at this round
        \COMMENT{late $\Rightarrow$ consensus iterate; early, well-cond.\ $\Rightarrow$ local track; early, ill-cond.\ $\Rightarrow$ no legacy (Rem.~\ref{rem:noanchor})}
\STATE \textbf{if} $y_j\neq\varnothing$: $g_j\leftarrow\nabla h_j(y_j)$;\quad
       $\eta_j\leftarrow\tfrac{a_j+b_j}{2}+\eta_j^{\mathrm{led}}$
        \COMMENT{midpoint of $[a_j,b_j]$ (Lem.~\ref{lem:center}): $\tfrac{\mu_j+L_j}{2}$ in (SC), $\tfrac{L_j}{2}$ in (C); $+$ inherited curvature (Prop.~\ref{prop:compose}); frozen from now on}
\STATE pick heir $a\in N_j$ with fewest legacies; send $(y_j,g_j,\eta_j,z_j^{t})$, or $(\varnothing,z_j^{t},\nabla h_j(x_j^{t}))$ if no legacy
        \COMMENT{one message: at most $3$ vectors $+1$ scalar; load balancing keeps $\hloc$ small}
\ENDIF
\IF{node $j$ receives $(y_i,g_i,\eta_i,z_i)$ from departing $i$}
\STATE $h_j\leftarrow h_j+q_i$ with $q_i$ from \eqref{eq:legacy};\quad
       $\eta_j^{\mathrm{led}}\!\leftarrow\!\eta_j^{\mathrm{led}}\!+\eta_i$
        \COMMENT{inherit the memory of $f_i$; $y_i,g_i,\eta_i$ are constants of $q_i$ forever}
\STATE $z_j\leftarrow z_j+z_i+\eta_i\,(x_j^{t}-y_i)$
        \COMMENT{exact tracker correction (Prop.~\ref{prop:invariant}); no restart needed}
\ENDIF
\IF{node $j$ receives $(\varnothing,z_i,\nabla h_i(x_i^{t}))$ from departing $i$}
\STATE $z_j\leftarrow z_j+z_i-\nabla h_i(x_i^{t})$
        \COMMENT{no-legacy hand-off: restores the tracking invariant on the survivors}
\ENDIF
\STATE all survivors adopt Metropolis weights of $G^{t+1}$
\end{algorithmic}
\vspace{2pt}\hrule
\end{minipage}
\end{figure*}

\begin{proposition}[Exact invariant preservation]\label{prop:invariant}
Line 13 of Algorithm~\ref{alg:lgt} preserves $\sum_{j\in S}z_j=\sum_{j\in S}\nabla h_j(x_j)$ exactly across a departure.
\end{proposition}
\begin{proof}
Removing $i$ deletes $z_i$ and $\nabla h_i(x_i)$ from the two sides of the invariant, and the heir must now carry $\nabla q_i(x_a)$ instead; the correction $z_a\!\leftarrow\! z_a+z_i-\nabla h_i(x_i^t)+\nabla q_i(x_a^t)$ restores the identity. Since $\nabla q_i(y_i)=g_i=\nabla h_i(y_i)=\nabla h_i(x_i^t)$ when $y_i=x_i^t$ (and $g_i$ substitutes for $\nabla h_i(x_i^t)$ otherwise, App.~\ref{app:sc}), it equals $z_a+z_i+\eta_i(x_a^t-y_i)$.
\end{proof}

\begin{remark}[Why a single heir]
If every survivor added a copy of the penalty, objective \eqref{eq:surrogate} would contain $|S^t|$ copies whose weights must be revised at every subsequent departure --- a silent re-weighting hazard. A single heir keeps the objective an unweighted sum, so any off-the-shelf decentralized method applies verbatim; Prop.~\ref{prop:compose} handles the heir's own eventual departure.
\end{remark}

\section{General Analysis and Main Theoretical Results}\label{sec:eta}

Every guarantee in this paper is an instantiation of one principle: \emph{any convergence guarantee for the surrogate $\Fh_k$ transfers to the original $F$ at the price of the residual terms \eqref{eq:gradtransfer}--\eqref{eq:valtransfer}, and the price is small exactly when the anchors are accurate.} The two regimes select the applicable clause: (SC) uses \eqref{eq:gradtransfer} at $x=\xs$ together with the strong convexity of $\Fh_k$ to bound \emph{distances} (Sec.~\ref{sec:theory}); (C) uses \eqref{eq:valtransfer} together with the strong convexity \emph{contributed by the legacies themselves} to bound \emph{values} (Sec.~\ref{sec:cvx}).

\subsection{The point-only variant and the schedule $\eta^t_i$.}
Suppose bandwidth or privacy constraints allow the farewell message to carry only the point $y=x^{t_m}_i$, with penalty $\frac{\eta^{(t_m)}_i}{2}\norm{x-y}^2$ (the squared form; a plain norm penalty destroys smoothness and acts as an exact penalty pinning the solution once the weight crosses a threshold). In App.~\ref{app:optimal_curvature} we show the bias of the resulting limit obeys
$\frac{\eta\varepsilon_t+\zeta_i}{(n-1)\mu+\eta}$
with $\varepsilon_t=\norm{x^t_i-\xs}$, that it improves on drop-and-forget \emph{iff} $\varepsilon_t<\zeta_i/((n-1)\mu)$, and that balancing the numerator yields the certified geometric schedule
\begin{equation}\label{eq:etat}
 \eta^{(t_m)}_i\;=\;\frac{\norm{\nabla f_i(x^{t_m}_i)}}{\hat\varepsilon(t_m)}
\;=\;\frac{\norm{\nabla f_i(x^{t_m}_i)}}{C\sqrt{n}\,\bar R}\,(1-\theta)^{-t_m/2}
\end{equation}
with resulting per-node bias $\le 2\hat\varepsilon(t_m)$. Gradient anchoring (one extra vector) improves this to $\frac{\kappa-1}{2n}\hat\varepsilon(t_m)$ --- a factor $\approx\frac{4n}{\kappa-1}$ --- and needs no error certificate; we therefore treat \eqref{eq:etat} as the fallback for bandwidth- or privacy-constrained exits.

\begin{proposition}[GT-P: rate and bias, regime (SC); proof in App.~\ref{app:optimal_curvature}]\label{prop:pointonly}
Run Algorithm~\ref{alg:lgt} with point-only legacies $\frac{\eta^{(t_m)}_i}{2}\norm{x-x^{t_m}_i}^2$, weights \eqref{eq:etat} capped at $\eta^{(t_m)}_i\le\frac{c_0}{\alpha}$ ($c_0$ an absolute constant). Then the iterates converge \emph{linearly}, at the rate of Proposition~\ref{fact:gt} with $\hloc$ replaced by $\hloc+\frac{c_0}{\alpha}$, to a point $\hat x$ with
\[
\norm{\hat x-\xs}\ \le\ \sum_{m=1}^{k}\min\Big\{2\hat\varepsilon(t_m),\ \tfrac{\hat\zeta_{i_m}}{(n-k)\mu}\Big\}\cdot(1+O(k\beta)) .
\]
The cap is not cosmetic: an uncapped weight eventually exceeds the smoothness budget of any fixed stepsize (Sec.~\ref{sec:exp} exhibits the divergence), while the cap induces the bias floor $\frac{\zeta_i}{(n-1)\mu+c_0/\alpha}$ for very late departures --- a dilemma the gradient-anchored legacy avoids entirely, since its weight never exceeds $L_i(1+\ell_{\max})$.
\end{proposition}

\begin{remark}[The weight is frozen at departure]\label{rem:frozen}
In \eqref{eq:etat} --- and equally in the anchor rule \eqref{eq:anchor} --- the superscript $(t_m)$ indexes the departure round: both the anchor $y_i$ and the weight $\eta^{(t_m)}_i$ are evaluated once, when node $i$ leaves, and the legacy keeps these constants for all subsequent rounds. Time-dependence across different departures is essential (later exits warrant stiffer penalties because the anchor is more accurate), but a weight that continued to grow after the departure would drive the surrogate toward the hard constraint $x=y$, pinning the limit to the anchor, discarding the survivors' information about $f_i$, and inflating the heir's smoothness constant without bound --- destroying the linear rate. The same convention applies to the gradient-anchored legacy, whose weight $\eta_i=\frac{\mu_i+L_i}{2}+\eta_i^{\mathrm{led}}$ does not depend on the departure round at all.
\end{remark}

\subsection{Case I: Strongly Convex Objectives}\label{sec:theory}
Throughout this section regime (SC) of Assumption~\ref{as:fun} is in force; the certified errors of Assumption~\ref{as:anchors} instantiate as $\varepsilon^{\mathrm{glob}}_m=\hat\varepsilon(t_m)$ and $\varepsilon^{\mathrm{loc}}_m=\hat\varepsilon^{\mathrm{loc}}(t_m)$ below.
Throughout, $\beta:=\frac{L-\mu}{2n\mu}=\frac{\kappa-1}{2n}$ (a dimensionless per-departure contraction) and $\theta=c\,\frac{(1-\lambda)^2}{\hat\kappa}$ is the gradient-tracking contraction with $\hat\kappa\le\kappa(1+\ell_{\max})$, $\ell_{\max}$ the maximum number of legacies per node.

\begin{proposition}[Gradient tracking, {\cite{nedic2017diging,qu2018harnessing}}]\label{fact:gt}
During any epoch (fixed graph, fixed objective $\Fh_m$) there is a Lyapunov function $\Psi^t$ combining optimality, consensus and tracking errors with $\Psi^{t+1}\le(1-\theta)\Psi^t$ and $\max_j\norm{x_j^t-x^{\star}_m}\le C_\Psi\sqrt{\Psi^t}$, provided $\alpha=\Theta\big(\frac{(1-\lambda)^2}{\hloc}\big)$.
\end{proposition}

\begin{theorem}[Strongly convex case; proof in App.~\ref{app:sc}]\label{thm:main}
Let Assumptions \ref{as:fun}(SC), \ref{as:net}, \ref{as:grace} hold, run Alg.~\ref{alg:lgt} with the stepsize of Proposition~\ref{fact:gt}, and suppose $k\beta\le\frac12$ (i.e., $k(\kappa-1)\le n$). Then for every $T\ge t_k$ and every surviving node $j$,
\begin{equation}\label{eq:main}
\norm{x_j^{T}-\xs}\;\le\;\underbrace{C\,(1-\theta)^{\frac{T-t_k}{2}}\,\bar R_k}_{\text{optimization to }x^{\star}_k}
\;+\;\underbrace{2\beta\sum_{m=1}^{k}\rho_m}_{\text{legacy bias}},
\end{equation}
where $\rho_m=\norm{y_{i_m}-x^{\star}_{m-1}}$ satisfies, under the adaptive anchor \eqref{eq:anchor},
\begin{equation}\label{eq:rho}
\rho_m\;\le\;\min\Big\{\hat\varepsilon(t_m),\ \tfrac{\zeta}{\mu}+\hat\varepsilon^{\mathrm{loc}}(t_m)\Big\}+2\beta\!\!\sum_{m'<m}\!\!\rho_{m'},
\end{equation}
$\hat\varepsilon(t)=C\sqrt{n}\bar R(1-\theta)^{t/2}$, $\hat\varepsilon^{\mathrm{loc}}(t)=(1-\kappa^{-1})^{t}\bar R$, and $\bar R_k\le C'(\bar R+\sum_m\rho_m)$. In particular, if consecutive departures are separated by $\Delta^{\star}=\frac{2}{\theta}\log(4C_JC_\Psi^2)$ rounds, then
\begin{multline}\label{eq:headline}
\norm{x_j^{T}-\xs}\le C(1-\theta)^{\frac{T-t_k}{2}}\bar R\\
+\tfrac{2C(\kappa-1)}{n}\min\Big[\sqrt{n}\bar R(1-\theta)^{\frac{t_1}{2}}\!;\! k\big(\tfrac{\zeta}{\mu}+\bar R(1-\kappa^{-1})^{t_1}\big)\Big].
\end{multline}
\end{theorem}

\begin{remark}[Exactness for isotropic quadratics]\label{cor:quad}
If every $f_i$ is quadratic with Hessian $\frac{\mu_i+L_i}{2}I$, then $\beta$'s role is played by $0$: Legacy-GT converges to $\xs$ exactly, for arbitrary departure times, while the classic algorithm keeps its full bias \eqref{eq:dropbias}.
\end{remark}

If $t_1\ge\frac{2}{\theta}\log\frac{C\sqrt n\bar R}{\epsilon}=O\big(\frac{\hat\kappa}{(1-\lambda)^2}\log\frac{\sqrt n\bar R}{\epsilon}\big)$, then the limit error of Legacy-GT is at most $\frac{(\kappa-1)k}{n}\epsilon\le\epsilon$, and the iteration complexity to reach $2\epsilon$ equals that of GT on a static network with no departures. Departed nodes are remembered up to the accuracy the collective had achieved when they left. Whatever $t_1$ is, the legacy bias never exceeds $\frac{(\kappa-1)k}{n}\cdot\frac{\zeta}{\mu}\cdot(1+o(1))$: the adaptive anchor caps the damage of an early exit at $\frac{\kappa-1}{2}$ times the classic per-node bias $\Theta(\frac{\zeta}{n\mu})$, while retaining the exponential improvement whenever departures are late. For $\kappa<3$ Legacy-GT dominates the classic algorithm. Finally, let us make the final remark that we will discuss in App.~\ref{app:sc}

\begin{corollary}[Selective memory]\label{cor:selective}
Under the three-way rule \eqref{eq:anchor}, Theorem~\ref{thm:main} holds with the $m$-th bias summand $2\beta\rho_m$ replaced by
$\min\big\{2\beta\hat\varepsilon(t_m),\ 2\beta(\zeta/\mu+\hat\varepsilon^{\mathrm{loc}}(t_m)),\ \tfrac{\hat\zeta_{i_m}}{(n-k)\mu}\big\}$:
every departure contributes the smallest of its three certified error bounds, so Legacy-GT is never worse than the classic algorithm by more than the certification slack.
\end{corollary}

\begin{theorem}[Lower bound for first-order memories; proof in App.~\ref{app:sc}]\label{thm:lower}
Fix $\mu<L$, $n$, and any map that compresses a departing node into a function of one first-order snapshot $(y,\nabla f_i(y))$. For every $\varepsilon>0$ there exist two instances, identical for the survivors and producing identical snapshots with $\norm{y-\xs}=\Theta(\varepsilon)$, whose true minimizers are $\frac{(L-\mu)\varepsilon}{(n-1)\mu+L}$ apart. Hence the worst-case bias of any such method is $\Omega\big(\frac{(\kappa-1)\varepsilon}{n+\kappa}\big)$, and Legacy-GT (bias $\le\frac{\kappa-1}{2n}\varepsilon$) is optimal up to a factor $O(\kappa)$.
\end{theorem}

\subsection{Case II: Convex Objectives ($\mu=0$)}\label{sec:cvx}
Regime (C): each $f_i$ is convex and $L_i$-smooth, $\omega_i=\frac{L_i}{2}$, $\eta_{i_m}=\frac{L_{i_m}}{2}+\eta^{\mathrm{led}}_{i_m}$, $X^{\star}=\argmin F\neq\emptyset$. Two structural facts change relative to (SC): the target is a set, so anchors must be certified against a common minimizer; and the survivors contribute no strong convexity --- all of it comes from the legacies.

\begin{assumption}[Common-minimizer anchors]\label{as:common}
There exists $\bar x^{\star}\in X^{\star}$ with $\norm{y_m-\bar x^{\star}}\le\varepsilon_m$ for all $m=1,\dots,k$. (Natural when anchors are iterates of one run converging to a single element of $X^{\star}$, as gradient methods do.)
\end{assumption}

\begin{theorem}[Convex case; proof in App.~\ref{app:con}]\label{thm:cvx}
Under Assumptions~\ref{as:fun}(C), \ref{as:net}, \ref{as:grace}, \ref{as:common}, Algorithm~\ref{alg:lgt} with consensus anchors enjoys:
\begin{itemize}
\item[(i)] \emph{(Departures create strong convexity.)} After $m\ge1$ departures, $\Fh_m$ is $\sigma_m$-strongly convex with $\sigma_m=\sum_{m'\le m}\eta_{i_{m'}}\ge \frac{m}{2}\min_i L_i$; in particular the base method converges linearly on every post-departure epoch, even though the original problem is merely convex.
\item[(ii)] \emph{(Value bias is quadratic in the anchor errors.)} The minimizer $\hat x$ of $\Fh_k$ satisfies $\norm{\hat x-\bar x^{\star}}\le\tilde\kappa\,\bar\varepsilon$ with $\tilde\kappa=\frac{\max_i L_i}{\min_i L_i}$, $\bar\varepsilon=\frac1k\sum_m\varepsilon_m$, and
\[
F(\hat x)-F^{\star}\ \le\ \frac{(2\tilde\kappa^2+3)\max_i L_i}{4}\cdot\frac{\big(\sum_{m}\varepsilon_m\big)^2}{k}\,.
\]
For a single departure: $F(\hat x)-F^{\star}\le\frac{(2\tilde\kappa^2+3)L}{4}\,\varepsilon_1^2$ --- {quadratically} small in the anchor accuracy, versus the {linear} distance bias of the (SC) case and the constant value gap $F(\xs_{S})-F^{\star}=\Theta(1)$ of the classic algorithm.
\end{itemize}
\end{theorem}

First, (i) is load-bearing for clause (ii): without strong convexity of $\Fh_k$ the surrogate minimizer $\hat x$ need not be unique nor stay near $\bar x^{\star}$, and the distance bound $\delta\le\tilde\kappa\bar\varepsilon$ would be unavailable. Second, it upgrades the rate: the classic algorithm on a merely convex problem converges sublinearly, at $O(1/t)$ in value, forever; under Legacy-GT the sublinear phase lasts only until the first departure, after which Proposition~\ref{fact:gt} applies with modulus $\sigma_m$ and convergence is linear. {Third, it explains the mechanism:} the departed curvature itself is the regularizer --- the legacies act as Tikhonov terms whose center and weight are chosen (Lemma~\ref{lem:center}) to {preserve} the solution rather than distort it, which is why the added strong convexity comes with a bias that is only quadratic in the anchor errors.

Finally, note, the quadratic improvement in (ii) is not an accident of the proof: the legacy matches $f_{i_m}$ to {first} order at $y_m$, so the value residual \eqref{eq:valtransfer} is second-order --- and, unlike in (SC), no first-order distance term survives because there is no fixed $\xs$ to drift from, only a value functional to control.

\begin{proposition}[Certificates in (C); proof in App.~\ref{app:con}]\label{prop:certC}
Suppose $F$ satisfies $\nu$-QG on the initial sublevel set: $F(x)-F^{\star}\ge\frac{\nu}{2}\,\mathrm{dist}(x,X^{\star})^2$. Let $G(t)$ be the value certificate of the base method, $F(\bar x^{t})-F^{\star}\le G(t)=O(1/t)$ for convex GT \cite{qu2018harnessing}. Then:
(a) $\mathrm{dist}(y_m,X^{\star})\le\sqrt{2G(t_m)/\nu}=:\gamma_m$;
(b) Assumption~\ref{as:common} holds with the concrete choices $\bar x^{\star}=\Pi_{X^{\star}}(y_k)$ and
$\varepsilon_m=\gamma_m+\norm{y_m-y_k}$,
where the second term is {directly observable};
(c) if $f_i$ itself satisfies $\nu_i$-QG, the local track certifies $\varepsilon^{\mathrm{loc}}_m=\sqrt{2G_i^{\mathrm{loc}}(t_m)/\nu_i}$ from the standard $O(1/t)$ value rate of gradient descent on convex $f_i$; a node lacking local QG simply never has its minimizer-anchor branch selected by the rule \eqref{eq:anchor}.
\end{proposition}

\section{Numerical Experiments}\label{sec:exp}
In this section we provide numerical experiements for considered methods.

\subsection{Quadratic function}
We validate the theory on heterogeneous quadratics, where every quantity in Theorem~\ref{thm:main} is computable in closed form. \textbf{Setup:} $n=20$, $d=10$, $f_i(x)=\frac12(x-c_i)^{\top}A_i(x-c_i)$ with $\mathrm{eig}(A_i)\subset[\mu,L]=[1,10]$, $\beta=\frac{\kappa-1}{2n}$, $c_i\sim\mathcal N(0,25 I)$; ring graph plus $15$ random chord; GT with $\alpha=0.02$;  initialization $x_j^0=0$. The classic baseline restarts trackers at each departure. The certificate $\hat\varepsilon(t)$ for the point-only weight \eqref{eq:etat} is calibrated from a pilot run ($R_0=3.27$, $\theta=0.074$), with the stability cap $\eta^{(t_m)}_i\le\frac{1}{2\alpha}$ mandated by the stepsize rule; anchor selection uses oracle distances (the certified rule \eqref{eq:anchor} selects identically here). We can see on Fig.~\ref{fig:traj} that Legacy-GT significantly outpeforms all other algorithms in  the sense of approached quality. At the same time, other algorithms gives convergence better too.

\begin{figure}[ht]
\centering
\includegraphics[width=0.5\columnwidth]{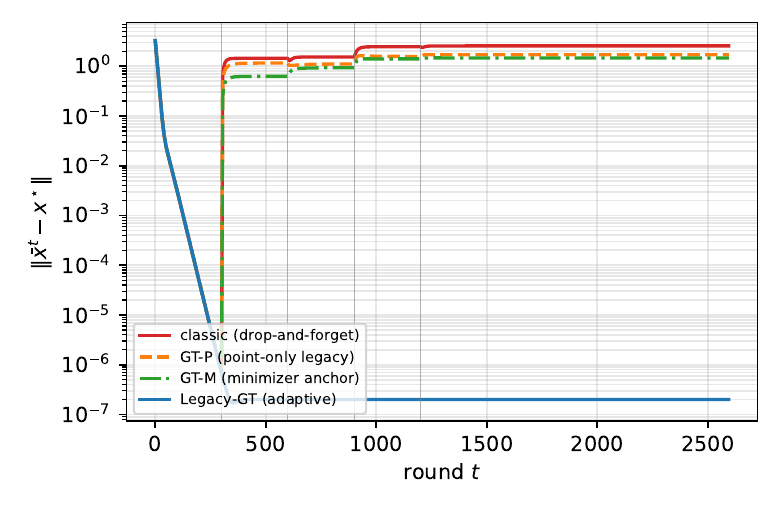}
\caption{Distance to the {true} optimum $\xs$ over $2600$ rounds; the four most heterogeneous nodes depart (gray lines). The classic algorithm and the two single-anchor variants plateau at their bias floors; Legacy-GT continues to $2.0\cdot10^{-7}$ --- seven orders of magnitude below the classic floor --- with only transient bumps at the departures, confirming that the exact tracker hand-off preserves the linear rate.}
\label{fig:traj}
\end{figure}
The idealized results above should not be extrapolated to all settings; we deliberately exhibit regime in which Legacy-GT does {not} reach high accuracy. Namely, we consider shared ill-conditioning and early mass departures (see Table~\ref{tab:stress}).

\begin{table}[ht]
\caption{Stress regime: quadratics with a {shared} eigenbasis (so the sum stays $\kappa$-ill-conditioned, unlike i.i.d.\ Hessians whose sum averages out), $k$ departures starting at round $t_1$ with gap $40$. Final errors after a long horizon.}
\label{tab:stress}
\centering
\begin{tabular}{|l|l|l|l|l|l|}
\hline
$\kappa$ & $k$ & $t_1$ & $k\beta$ & classic / Legacy-GT & gain \\
\hline
$10$  & $4$  & $50$  & $0.9$  & $3.2$ / $5.6\cdot10^{-2}$ & $57\times$ \\
$30$  & $8$  & $50$  & $5.8$  & $2.7$ / $3.2\cdot10^{-1}$ & $8.5\times$ \\
$100$ & $8$  & $50$  & $19.8$ & $2.7$ / $6.6\cdot10^{-1}$ & $4.1\times$ \\
$100$ & $12$ & $50$  & $29.7$ & $3.9$ / $6.6\cdot10^{-1}$ & $5.8\times$ \\
$100$ & $12$ & $800$ & $29.7$ & $3.9$ / $1.2\cdot10^{-1}$ & $33\times$ \\
\hline
\end{tabular}

\end{table}
 The gain over classic degrades monotonically in $k\beta$. Two features of the Table~\ref{tab:stress} are informative (consistent with Cor.~\ref{cor:selective}) and the departures is the dominant variable, exactly as the $\hat\varepsilon(t_1)$ factor in \eqref{eq:headline} predicts.
 
\subsection{Real datasets}\label{sec:exp_real}

\begin{figure}[ht]
\centering
\includegraphics[width=0.99\columnwidth]{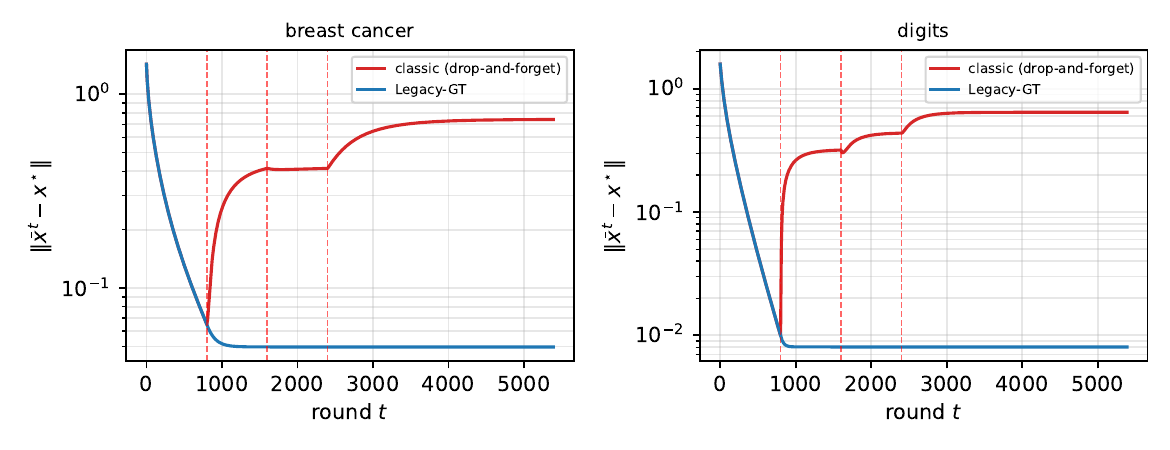}
\caption{Logistic regression ($\ell_2$-regularized) on two real datasets "breast cancer" and "digits", split non-IID across $n=10$ nodes (sorted by label and leading principal component, then chunked). The three most heterogeneous nodes depart permanently. Left: breast cancer ($569\times30$, estimated $\hat\kappa_i$ up to $322$). Right: digits, odd-vs-even ($1797\times64$, $\hat\kappa_i$ up to $157$).}
\label{fig:real}
\end{figure}

Further, let us make three observations. \emph{(a)} Each departure injects a visible jump in the classic trajectory while Legacy-GT continues toward $\xs$, ending closer. \emph{(b)} The bias ordering across the three leavers follows the theory: the first departure dominates Legacy-GT's final error, exactly the $\hat\varepsilon(t_1)$ factor in \eqref{eq:headline}. \emph{(c)} The final gaps are far from machine precision: with high $\hat\kappa_i=100$ and $\beta_i$ exceeds the smallness threshold --- consistent with the stress results of Table~\ref{tab:stress}.

\section{Conclusion}

All considered methods converge {linearly at the same rate} $\theta$; they differ in the limit point, as summarized in Table~\ref{tab:compare} (Sec.~1); recall $\varepsilon_m:=\hat\varepsilon(t_m)$ decays geometrically in the departure round.

Further, let us emphasize key difference with the classic algorithm.
(i) \emph{Late regime} ($t_1\gtrsim\theta^{-1}\log n$): classic bias stays $\Theta(k\zeta/(n\mu))$; Legacy-GT bias is $\frac{(\kappa-1)k}{n}\varepsilon(t_1)$. The crossover happens as soon as the network error drops below $\approx\frac{2\zeta}{(\kappa-1)\mu}$. (ii) \emph{Early regime}: the adaptive anchor falls back to the local minimizer track; per-node bias $\le\frac{L-\mu}{2}\cdot\frac{\zeta/\mu}{n\mu}$ versus classic $\frac{\zeta_i}{(n-k)\mu}$ with $\zeta_i\in[\mu d_i,Ld_i]$; the ratio guaranteed at least $2\times$ improvement whenever the dropped node's gradient at $\xs$ holds $\zeta_i\approx L d_i$. (iii) \emph{Optimality}: by Theorem~\ref{thm:lower} no first-order memory can improve on Legacy-GT by more than a factor $O(\kappa)$, so the compensation achieved is essentially all the compensation available at $O(d)$ farewell bandwidth; closing the remaining $\kappa$ gap provably requires second-order information (a Hessian sketch), trading bandwidth $O(rd)$ for bias --- we leave this to future work.

The three variants in Table~\ref{tab:compare} isolate the value of each design choice. The frozen weight \eqref{eq:etat} alone (point-only) already makes the bias decay with the departure time --- with any fixed weight it would not vanish at all (App.~\ref{app:optimal_curvature}). Adding the gradient term improves by factor $\frac{4n}{\kappa-1}$ and removes the dependence on the error certificate. Finally, the heir mechanism and the exact tracker hand-off contribute no bias at all --- their role is to preserve the linear rate and the unweighted-sum structure across departures.

Forgetting departed nodes costs decentralized optimization a permanent, heterogeneity-sized bias. We showed that a single farewell message --- an adaptively chosen anchor, the local gradient there, and the midpoint curvature $\eta_i=\frac{\mu_i+L_i}{2}$ preserves a departing node's influence up to $\frac{\kappa-1}{2n}$ {times the smaller of (its optimization error at departure) and its heterogeneity radius)}, composes losslessly across chained departures, keeps the exact linear rate of GT, and is optimal up to a $\kappa$ factor among all first-order memories.
\bibliographystyle{splncs04}
\bibliography{mybibliography}

\clearpage
\onecolumn
\appendix

\section{Missing Proofs for Section~\ref{sec:eta} (General Analysis and Main Theoretical Results)} \label{app:gen_main}

\subsection{Optimal Curvature}\label{sec:meta}
We start from some auxilliary Lemmas to prove our result.
\begin{lemma}[Curvature centering, general form; proof in App.~\ref{app:optimal_curvature}]\label{lem:center}
Let $h$ have curvature in $[a,b]$ (Assumption~\ref{as:fun}) and let $q$ be its legacy \eqref{eq:legacy} at anchor $y$ with $\eta\in[a,b]$. Then for all $x$,
\begin{equation}\label{eq:center}
\norm{\nabla q(x)-\nabla h(x)}\;\le\;\max\{b-\eta,\ \eta-a\}\,\norm{x-y}.
\end{equation}
The bound is minimized by $\eta^{\star}=\frac{a+b}{2}$, giving the curvature spread $\omega=\frac{b-a}{2}$: $\frac{L_h-\mu_h}{2}$ in (SC), $\frac{L_h}{2}$ in (C). 
\end{lemma}

Note, in (SC), $\eta^{\star}$ also minimizes the induced bias $\max\{L-\eta,\eta-\mu\}\norm{y-\xs}/((n-1)\mu+\eta)$.

\begin{lemma}[Bias transfer; proof in App.~\ref{app:optimal_curvature}]\label{lem:transfer}
Let $r_m:=h_{i_m}-q_{i_m}$ denote the $m$-th surrogate residual, so that $F=\Fh_k+\sum_{m=1}^{k}r_m$. Under Assumption~\ref{as:fun} with weights $\eta_{i_m}=\frac{a_{i_m}+b_{i_m}}{2}+\eta^{\mathrm{led}}_{i_m}$, for every $x$:
\begin{align}
\norm{\nabla F(x)-\nabla\Fh_k(x)}&\le\sum_{m=1}^{k}\omega_{i_m}\,\norm{x-y_m},\label{eq:gradtransfer}\\
\big|F(x)-\Fh_k(x)\big|&\le\sum_{m=1}^{k}\tfrac{\omega_{i_m}}{2}\,\norm{x-y_m}^2.\label{eq:valtransfer}
\end{align}
\end{lemma}

\begin{proposition}[Lossless composition; proof in App.~\ref{app:optimal_curvature}]\label{prop:compose}
If $h_i=f_i+Q$ where $Q$ is a quadratic of curvature $\eta_Q I$ (the inherited legacies), then the legacy of $h_i$ with $\eta_i=\frac{\mu_i+L_i}{2}+\eta_Q$ satisfies $\norm{\nabla q_i(x)-\nabla h_i(x)}\le\frac{L_i-\mu_i}{2}\norm{x-y_i}$: the inherited part is reproduced \emph{exactly}. Hence chained departures accumulate surrogate error only once per original function.
\end{proposition}
\subsection{Proofs for Section~\ref{sec:meta} (Optimal Curvature and Bias Transfer)} \label{app:optimal_curvature}

\textbf{Proof of Lemma \ref{lem:center}.}
Let $h$ have curvature in $[a,b]$. Write $h_\eta(x):=h(x)-\frac{\eta}{2}\norm{x}^2$, so that
$\nabla q(x)-\nabla h(x)=-\big[\nabla h_\eta(x)-\nabla h_\eta(y)\big]$,
and it suffices to show that $\nabla h_\eta$ is $\max\{|a|,|b|\}$-Lipschitz with $a:=a_h-\eta\le 0\le b:=b_h-\eta$ (for $\eta\in[a_h,b_h]$).

By assumption, $h_\eta$ has curvature in $[a,b]$ with $a\le0\le b$. Define $\varphi:=h_\eta-\frac{a}{2}\norm{\cdot}^2$; then $\varphi$ is convex and $(b-a)$-smooth. Fix $x,y$ and set
\[
u:=\nabla\varphi(x)-\nabla\varphi(y),\qquad v:=x-y,
\]
so that $\nabla h_\eta(x)-\nabla h_\eta(y)=u+av$. Convex $(b-a)$-smooth functions satisfy the Baillon--Haddad (co-coercivity) inequality and the Lipschitz bound:
\begin{equation}\label{eq:bh}
\inner{u}{v}\ \ge\ \frac{\norm{u}^2}{b-a},
\qquad
\norm{u}\ \le\ (b-a)\norm{v}.
\end{equation}
Expand:
\begin{gather*}
\norm{u+av}^2=\norm{u}^2+2a\inner{u}{v}+a^2\norm{v}^2
\ \le\ \norm{u}^2+\frac{2a}{b-a}\norm{u}^2+a^2\norm{v}^2
\\
=\frac{b+a}{b-a}\,\norm{u}^2+a^2\norm{v}^2 \leq 
(b+a)(b-a)\norm{v}^2+a^2\norm{v}^2=b^2\norm{v}^2,
\end{gather*}

where the inequality used $a\le0$ together with the \emph{lower} bound in \eqref{eq:bh}. Two cases:
\begin{itemize}
\item If $a+b\le0$ (i.e., $\eta\ge\frac{\mu_h+L_h}{2}$) $\Rightarrow \norm{u+av}^2\le a^2\norm{v}^2$.
\item If $a+b\ge0$, hence $\norm{u+av}^2\le b^2\norm{v}^2$.
\end{itemize}
In both cases $\norm{u+av}\le\max\{|a|,|b|\}\norm{v}=\max\{\eta-a_h,\ b_h-\eta\}\norm{v}$, proving \eqref{eq:center}. The maximum of the two linear functions of $\eta$ is minimized at their intersection $\eta^{\star}=\frac{a_h+b_h}{2}$, with value $\omega=\frac{b_h-a_h}{2}$.

\textbf{Proof of Lemma \ref{lem:transfer}.}
By Proposition~\ref{prop:compose}, $\nabla r_m=\nabla h_{i_m}-\nabla q_{i_m}$ satisfies $\norm{\nabla r_m(x)}\le\omega_{i_m}\norm{x-y_m}$ for all $x$; summing over $m$ and using $F=\Fh_k+\sum_m r_m$ gives \eqref{eq:gradtransfer}. For \eqref{eq:valtransfer}: $r_m(y_m)=0$ and $\nabla r_m(y_m)=0$ by construction of the legacy (it matches $h_{i_m}$ to first order at $y_m$), so
$|r_m(x)|=\big|\int_0^1\inner{\nabla r_m(y_m+s(x-y_m))}{x-y_m}\,ds\big|
\le\int_0^1 \omega_{i_m}\,s\norm{x-y_m}^2\,ds=\tfrac{\omega_{i_m}}{2}\norm{x-y_m}^2$;
summing over $m$ proves the claim. \hfill$\blacksquare$

For the second claim, consider (for a single departure; the general case is identical with $\hat\mu_m$ in place of $(n-1)\mu+\eta$) the bias functional
$\phi(\eta)=\frac{\max\{L-\eta,\eta-\mu\}}{(n-1)\mu+\eta}\,\norm{y-\xs}$.
On $[\mu,\frac{\mu+L}{2}]$ the numerator $L-\eta$ decreases while the denominator increases, so $\phi$ decreases.
On $[\frac{\mu+L}{2},L]$,
\[
\frac{d}{d\eta}\,\frac{\eta-\mu}{(n-1)\mu+\eta}
=\frac{(n-1)\mu+\mu}{\big((n-1)\mu+\eta\big)^2}>0,
\]
so $\phi$ increases. Hence $\eta^{\star}=\frac{\mu+L}{2}$ minimizes the bias as well: the extra strong convexity gained beyond the midpoint is a second-order effect dominated by the first-order growth of the surrogate error. \hfill$\blacksquare$

\textbf{Proof of Proposition~\ref{prop:compose}.}
Let $h_i=f_i+Q$ with $Q(x)=\frac{\eta_Q}{2}\norm{x}^2+\inner{p}{x}+c$ (any quadratic with Hessian $\eta_Q I$; inherited legacies sum to this form). The legacy of $h_i$ at anchor $y$ with curvature $\eta_i=\eta_f+\eta_Q$, $\eta_f:=\frac{\mu_i+L_i}{2}$, has gradient
\[
\nabla q_i(x)=\nabla h_i(y)+(\eta_f+\eta_Q)(x-y)
=\underbrace{\big[\nabla f_i(y)+\eta_f(x-y)\big]}_{\text{legacy of }f_i}
+\underbrace{\big[\nabla Q(y)+\eta_Q(x-y)\big]}_{=\ \nabla Q(x)\ \text{exactly}},
\]
because $\nabla Q$ is affine with slope $\eta_Q I$. Hence $\nabla q_i-\nabla h_i=\big[\nabla f_i(y)+\eta_f(\cdot-y)\big]-\nabla f_i$, and Lemma~\ref{lem:center} applied to $f_i$ alone gives $\norm{\nabla q_i(x)-\nabla h_i(x)}\le\frac{L_i-\mu_i}{2}\norm{x-y}$. \hfill$\blacksquare$

\textbf{Derivation of the point-only schedule \eqref{eq:etat}.}
Let the farewell carry only $y=x^t_i$, the penalty be $q^{\mathrm{pt}}(x)=\frac{\eta}{2}\norm{x-y}^2$, and consider a single departure ($k=1$; the general case adds the compounding of App.~\ref{app:sc}). The perturbed problem $\Fh=\sum_{j\neq i}f_j+q^{\mathrm{pt}}$ is $((n-1)\mu+\eta)$-strongly convex, and
\[
\nabla\Fh(\xs)=\eta(\xs-y)-\nabla f_i(\xs)
\ \Longrightarrow\
\norm{\hat x-\xs}\le\frac{\eta\,\varepsilon_t+\zeta_i}{(n-1)\mu+\eta},
\quad \varepsilon_t=\norm{y-\xs},\ \zeta_i=\norm{\nabla f_i(\xs)}.
\]
The map $\eta\mapsto\frac{\eta\varepsilon_t+\zeta_i}{(n-1)\mu+\eta}$ has derivative proportional to $\varepsilon_t(n-1)\mu-\zeta_i$: if $\varepsilon_t>\frac{\zeta_i}{(n-1)\mu}$ (early departure) the bias is minimized at $\eta=0$, i.e., \emph{no} point-only memory beats drop-and-forget; if $\varepsilon_t<\frac{\zeta_i}{(n-1)\mu}$ it decreases toward the limit $\varepsilon_t$ as $\eta\to\infty$ --- but $\eta=\infty$ pins the solution to $y$, discards the survivors' information, and is inconsistent under multiple departures. Balancing the numerator ($\eta\varepsilon_t\approx\zeta_i$) with the observable proxies $\zeta_i\approx\norm{\nabla f_i(y)}$ (error $O(L\varepsilon_t)$) and the certificate $\varepsilon_t\le\hat\varepsilon(t)$ gives \eqref{eq:etat}; substituting back,
\[
\norm{\hat x-\xs}\ \le\ \frac{2\,\norm{\nabla f_i(y)}}{(n-1)\mu+\norm{\nabla f_i(y)}/\hat\varepsilon(t)}\ \le\ 2\,\hat\varepsilon(t),
\]
with the last step valid once $\eta^t_i\ge(n-1)\mu$, i.e., once the departure is late enough for memory to pay at all. Comparing with the gradient-anchored bias $\frac{\kappa-1}{2n}\hat\varepsilon(t)$ (Lemmas~\ref{lem:center}, \ref{lem:drift}) exhibits the $\frac{4n}{\kappa-1}$ gap quoted in the main text: encoding the \emph{direction} $\nabla f_i(y)$ in the penalty's linear term, rather than fighting it with the penalty's strength, is what buys the factor $n$.

\textbf{Completing Proposition \ref{prop:pointonly}: the cap and the rate.}
The surrogate objective of GT-P is $\big((n-1)\mu+\eta\big)$-strongly convex and its heir's working function is $\big(L_{a}+\eta\big)$-smooth. Proposition~\ref{fact:gt} therefore applies verbatim with $\hloc$ replaced by $\hloc+\eta$; the stepsize rule $\alpha=\Theta\big(\frac{(1-\lambda)^2}{\hloc+\eta}\big)$ is consistent with a \emph{fixed} $\alpha$ only if $\eta\le\frac{c_0}{\alpha}$ for an absolute $c_0$ --- hence the cap, and hence the divergence observed in Sec.~\ref{sec:exp} when the cap is removed and $\eta^{(t_m)}_i\propto(1-\theta)^{-t_m/2}$ outgrows it. Under the cap, linear convergence to the surrogate optimum holds at the stated rate, and the bias is the minimum of: the balanced bound $2\hat\varepsilon(t_m)$ when \eqref{eq:etat} is feasible within the cap, the capped bound $\frac{\zeta_i}{(n-1)\mu+c_0/\alpha}$ otherwise, and the drop bound $\frac{\hat\zeta_i}{(n-k)\mu}$ if the weight is set to zero; multiple departures compound exactly as in the compounding bound of A.\ref{app:sc}. \hfill$\blacksquare$

\subsection{Technical lemmas for Strong Convex Case}
\label{app:tech_lammas}
\begin{lemma}[Uniform strong convexity; proof in App.~\ref{app:tech_lemmas_proof}]\label{lem:mu}
With curvatures chosen as in Alg.~\ref{alg:lgt}, $\Fh_m$ is $n\mu$-strongly convex for all $m$.
\end{lemma}

\begin{lemma}[Per-departure drift; proof in App.~\ref{app:tech_lemmas_proof}]\label{lem:drift}
It holds that
$\norm{x^{\star}_m-x^{\star}_{m-1}}\le\beta\,\norm{y_{i_m}-x^{\star}_{m-1}}$.
\end{lemma}

\begin{lemma}[Boundary jump; proof in App.~\ref{app:tech_lemmas_proof}]\label{lem:jump}
At a departure, re-averaging over $n-m$ nodes, applying the tracker correction of Alg.~\ref{alg:lgt} (line 12), and shifting the target from $x^{\star}_{m-1}$ to $x^{\star}_m$ inflate the Lyapunov function by at most an absolute constant: $\Psi_{\mathrm{new}}\le C_J\,\Psi_{\mathrm{old}}$.
\end{lemma}

\subsection{Proofs for Section~\ref{sec:theory} (Case I: Strongly Convex Objectives)} \label{app:sc}

\textbf{A.2.1\quad Setting and the Lyapunov function}
Fix an epoch with survivor set $S$, $|S|=N$, objective $\Fh$ with minimizer $x^{\star}_{\mathrm{ep}}$, local working functions $h_j$ that are $\mu_j^{h}$-strongly convex and $L_j^{h}$-smooth with $L_j^h\le\hloc:=L(1+\ell_{\max})$. Denote $\bar x=\frac1N\sum_{j\in S}x_j$, $\bar z=\frac1N\sum_j z_j$, and
\[
D:=\norm{\bar x-x^{\star}_{\mathrm{ep}}}^2,\qquad
V_x:=\sum_{j\in S}\norm{x_j-\bar x}^2,\qquad
V_z:=\sum_{j\in S}\norm{z_j-\bar z}^2,
\]
\[
\Psi\ :=\ D+c_1 V_x+c_2\,\alpha^2 V_z .
\]
Proposition~\ref{fact:gt} (the standard small-gain analysis of gradient tracking, e.g., \cite{qu2018harnessing}, Sec.~IV, or \cite{nedic2017diging}, Thm.~3.2) provides absolute constants $c_1,c_2$ and a stepsize $\alpha=\Theta\big(\frac{(1-\lambda)^2}{\hloc}\big)$ such that within the epoch $\Psi^{t+1}\le(1-\theta)\Psi^t$ with $\theta=c\frac{(1-\lambda)^2}{\hat\kappa}$, $\hat\kappa=\hloc/\mu$, and moreover $\norm{x_j-x^{\star}_{\mathrm{ep}}}^2\le 2D+2V_x\le C_\Psi^2\Psi$ for all $j$, with $C_\Psi^2=2\max\{1,c_1^{-1}\}$. We take the stepsize and $(c_1,c_2)$ tuned once for the worst constants over all epochs ($\lambda$, $\hloc$), which is legitimate since the guarantee is monotone in these parameters.

\textbf{A.2.3\quad Proof of Theorem \ref{thm:main}} \label{app:theorem_proof}
Let $\psi_m:=\sqrt{\Psi\ \text{at }t_m\text{ (w.r.t.\ }x^{\star}_{m-1},S_{m-1})}$ and $\Delta_m:=t_m-t_{m-1}$ ($t_0=0$). Proposition~\ref{fact:gt} within epochs and \eqref{eq:jump} at boundaries give the scalar recursion
\begin{equation}\label{eq:rec}
\psi_{m}\ \le\ (1-\theta)^{\Delta_m/2}\,\sqrt{C_J}\,\big(\psi_{m-1}+\beta a_{m-1}\big),
\qquad \psi_0\le C\sqrt{n}\,\bar R,
\end{equation}
where $\bar R^2$ bounds $\frac1n\Psi^0$ under arbitrary initialization ($\bar R\le R+\alpha\max_j\norm{\nabla h_j(x_j^0)}\le R+O(\zeta/\hloc+R)$; with identical initialization $x_j^0\equiv x^0$ the consensus part vanishes and only the tracker part contributes). The anchor error at departure $m$ satisfies, by Step 3 of A.\ref{app:theorem_proof} and the anchor rule \eqref{eq:anchor},
\begin{equation}\label{eq:rhobound}
\rho_m\ \le\ \min\Big\{\,C_\Psi\psi_m,\ \ \tfrac{\zeta}{\mu}+\hat\varepsilon^{\mathrm{loc}}(t_m)+B_{m-1}\Big\},
\qquad B_{m-1}\le\beta\!\!\sum_{m'<m}\!\!\rho_{m'}\ \ \text{(see the compounding bound above)},
\end{equation}
and $C_\Psi\psi_m\le\hat\varepsilon(t_m):=C\sqrt n\bar R(1-\theta)^{t_m/2}\cdot\prod_{m'\le m}\sqrt{C_J}(1+\beta a_{m'-1}/\psi_{m'-1})$; under the spacing condition $\Delta_m\ge\Delta^{\star}=\frac2\theta\log(4C_JC_\Psi^2)$ each factor $(1-\theta)^{\Delta_m/2}\sqrt{C_J}\le\frac12 C_\Psi^{-1}\cdot C_\Psi\le\frac12$, so $\psi_m\le\frac12(\psi_{m-1}+\beta a_{m-1})$ and the product telescopes into the constant absorbed in $C$; this justifies the form of $\hat\varepsilon(t)$ used in the statement (without spacing, apply \eqref{eq:rec} verbatim, which only improves the exponent to $t_m$ in place of $t_1$ but with constants $C_J^{m/2}$). This establishes \eqref{eq:rho}.

Finally, for $T\ge t_k$, one last application of Proposition~\ref{fact:gt} on the terminal epoch gives
$\norm{x_j^T-x^{\star}_k}\le C_\Psi(1-\theta)^{(T-t_k)/2}\psi_k^{+}$ with $\psi_k^+\le\sqrt{C_J}(\psi_k+\beta a_k)=:C\bar R_k$, and the triangle inequality with the compounding bound \eqref{eq:compound},
\[
\norm{x_j^T-\xs}\ \le\ \norm{x_j^T-x^{\star}_k}+B_k
\ \le\ C(1-\theta)^{\frac{T-t_k}{2}}\bar R_k+\beta\sum_{m=1}^{k}\rho_m ,
\]
proves \eqref{eq:main} (the factor $2$ in the theorem covers the $a_m$-vs-$\rho_m$ bookkeeping via $k\beta\le\frac12$ and \eqref{eq:rhobound}). The headline bound \eqref{eq:headline} follows by inserting \eqref{eq:rhobound}: the consensus branch gives $\sum_m\rho_m\le 2C_\Psi\psi_1\le C\sqrt n\bar R(1-\theta)^{t_1/2}$ under spacing (geometric halving), while the local branch gives $\sum_m\rho_m\le 2k(\frac{\zeta}{\mu}+\bar R(1-\kappa^{-1})^{t_1})$ using $k\beta\le\frac12$ to unwind the $B_{m-1}$ terms; taking the minimum termwise yields \eqref{eq:headline}. \hfill$\blacksquare$

\textbf{A.2.4\quad Comments on tightness}
The $\sqrt n$ in $\hat\varepsilon(t)$ is an artifact of allowing adversarial initialization inside the Lyapunov function; with identical initialization and homogeneous gradients it reduces to $O(1)$. The condition $k\beta\le\frac12$ ($k(\kappa-1)\le n$) is only used to keep compounding constants below $2$; beyond it, every bound degrades gracefully by $e^{k\beta}$. The spacing condition $\Delta^{\star}=O(1/\theta)$ matches the intuition that the network needs one mixing-and-contraction period to ``absorb'' a departure; clustered departures are covered by the non-spaced form of \eqref{eq:rec}.

\textbf{Proof of Theorem \ref{thm:lower} (lower bound).}
Consider algorithms whose retained memory of the departing node is an arbitrary measurable function $M(y,v)$ of one first-order snapshot $y\in\R^d$, $v=\nabla f_i(y)$, and whose output after the departure is any function of $M(y,v)$ and the survivors' data. Fix the survivors' aggregate as $F_{\mathrm{rest}}(x)=\frac{(n-1)\mu}{2}\norm{x-b}^2$ (admissible under Assumption~\ref{as:fun}). Work on the line $d=1$ and fix a snapshot $(y,v)=(0,\gamma)$, $\gamma\neq0$.

Both of the following are $\mu$-strongly convex, $L$-smooth, and produce exactly this snapshot:
\[
f(x)=\tfrac{\mu}{2}(x-a)^2,\ \ a=-\tfrac{\gamma}{\mu};
\qquad
\tilde f(x)=\tfrac{L}{2}(x-\tilde a)^2,\ \ \tilde a=-\tfrac{\gamma}{L}.
\]
The corresponding global minimizers of $F_{\mathrm{rest}}+f$ and $F_{\mathrm{rest}}+\tilde f$ are
\[
\xs=\frac{(n-1)\mu\, b-\gamma}{(n-1)\mu+\mu},
\qquad
\tilde x^{\star}=\frac{(n-1)\mu\, b-\gamma}{(n-1)\mu+L}.
\]
Choose $b,\gamma$ so that the snapshot sits at distance $\varepsilon$ from the first instance's optimum: $|\xs-y|=|\xs|=\varepsilon$, i.e., $|(n-1)\mu b-\gamma|=\varepsilon\,((n-1)\mu+\mu)$. Then
\[
|\xs-\tilde x^{\star}|
=\big|(n-1)\mu b-\gamma\big|\cdot\frac{L-\mu}{((n-1)\mu+\mu)((n-1)\mu+L)}
=\frac{(L-\mu)\,\varepsilon}{(n-1)\mu+L},
\]
and $|\tilde x^{\star}-y|=\Theta(\varepsilon)$ as well. The two instances are indistinguishable to the algorithm (identical survivors, identical memory $M(0,\gamma)$), so its limit point is the same on both; on at least one instance the error is at least $\frac12|\xs-\tilde x^{\star}|=\frac{(L-\mu)\varepsilon}{2((n-1)\mu+L)}=\Omega\big(\frac{(\kappa-1)\varepsilon}{n+\kappa}\big)$. Embedding the construction in the first coordinate extends it to any $d$, and adding $\frac{\mu}{2}\norm{\cdot}^2$-type padding to the survivor pool extends it to any admissible $F_{\mathrm{rest}}$ spectrum. Comparing with the upper bound $\frac{(\kappa-1)\varepsilon}{2n}$ of Legacy-GT completes the proof; the gap is a single factor $\frac{n\kappa+\kappa^2}{n+\kappa}/n\le O(\kappa)$. \hfill$\blacksquare$

\begin{remark}
The restriction to snapshot-based memories is exactly the algorithm class of this paper. Removing it would require the adversary to remain consistent with the entire pre-departure trajectory; for quadratics, $d+1$ generic first-order queries identify the function, so trajectory-consistent lower bounds need non-quadratic perturbation families (bump constructions vanishing to first order on the trajectory). We conjecture the same $\Omega\big(\frac{(\kappa-1)\varepsilon}{n\kappa}\big)$ floor and leave it open.
\end{remark}

\textbf{Proof of Corollary \ref{cor:selective}.}
Let $B_m:=\norm{x^{\star}_m-\xs}$ and $\rho_m:=\norm{y_{i_m}-x^{\star}_{m-1}}$. Lemma~\ref{lem:drift} and the triangle inequality give $B_m\le B_{m-1}+\beta\rho_m$, hence
\begin{equation}\label{eq:compound}
B_k\ \le\ \beta\sum_{m=1}^{k}\rho_m .
\end{equation}
If instead one measures anchors against the \emph{true} optimum, $a_m:=\norm{y_{i_m}-\xs}$, then $\rho_m\le a_m+B_{m-1}$, so $B_m\le(1+\beta)B_{m-1}+\beta a_m$ and, unrolled, $B_k\le\beta\sum_m(1+\beta)^{k-m}a_m\le e^{k\beta}\beta\sum_m a_m\le 2\beta\sum_m a_m$ whenever $k\beta\le\frac12$. Also note for later use, $\rho_m\le a_m+B_{m-1}\le a_m+\beta\sum_{m'<m}\rho_{m'}$, which is inequality \eqref{eq:rho} of the main text once $a_m$ is bounded by the anchor rule. 

If departure $m$ chooses an anchored legacy, its contribution to $B_k$ is $2\beta\rho_m$ with $\rho_m$ bounded by the corresponding branch of \eqref{eq:rhobound}, as in A.\ref{app:theorem_proof}. If it chooses no legacy, then $\Fh_m=\Fh_{m-1}-h_{i_m}$, and
$\norm{\nabla\Fh_m(x^{\star}_{m-1})}=\norm{\nabla h_{i_m}(x^{\star}_{m-1})}\le\zeta_{i_m}+L_{i_m}\big(\norm{y}-\text{terms}\big)\le\hat\zeta_{i_m}+L_{i_m}B_{m-1}$,
while $\Fh_m$ is $(n-k)\mu$-strongly convex (only the surviving indices contribute), so the drift is at most $\frac{\hat\zeta_{i_m}}{(n-k)\mu}+O(\kappa B_{m-1})$; unrolling as in \eqref{eq:compound} under $k\beta\le\frac12$ (and $k\kappa/(n-k)$ bounded, absorbed in the constant) replaces the $m$-th summand accordingly. Since the rule selects the branch with the smallest certified bound, each summand is the minimum of the three. \hfill$\blacksquare$

\subsection{Proofs for Appendix~\ref{app:tech_lammas}} \label{app:tech_lemmas_proof}

\textbf{Proof of Lemma \ref{lem:mu}.}
By induction on $m$: each legacy issued in Alg.~\ref{alg:lgt} carries curvature $\eta_i=\frac{\mu_i+L_i}{2}+\eta_i^{\mathrm{led}}\ge\mu_i+\eta_i^{\mathrm{led}}$, and $\eta_i^{\mathrm{led}}$ is (inductively) at least the sum of $\mu_{i'}$ over all original indices $i'$ folded into node $i$'s inherited legacies. Hence in $\Fh_m$ every original index $i'\in[n]$ contributes at least $\mu_{i'}\ge\mu$ of curvature, either through its surviving $f_{i'}$ or through the curvature ledger of exactly one legacy chain. Summing, $\Fh_m$ is $\sum_{i'}\mu_{i'}\ge n\mu$-strongly convex. \hfill$\blacksquare$

\textbf{Proof of Lemma \ref{lem:drift}.}
$\Fh_m=\Fh_{m-1}-h_{i_m}+q_{i_m}$, hence, using $\nabla\Fh_{m-1}(x^{\star}_{m-1})=0$ and Proposition~\ref{prop:compose},
\[
\norm{\nabla\Fh_m(x^{\star}_{m-1})}
=\norm{\nabla q_{i_m}(x^{\star}_{m-1})-\nabla h_{i_m}(x^{\star}_{m-1})}
\le\tfrac{L-\mu}{2}\,\norm{y_{i_m}-x^{\star}_{m-1}}.
\]
By Lemma~\ref{lem:mu}, $\Fh_m$ is $n\mu$-strongly convex, and strong convexity yields, for any $x$,
$\norm{x-x^{\star}_m}\le\norm{\nabla\Fh_m(x)}/(n\mu)$ (take $x=x^{\star}_{m-1}$):
indeed $n\mu\norm{x-x^{\star}_m}^2\le\inner{\nabla\Fh_m(x)-\nabla\Fh_m(x^{\star}_m)}{x-x^{\star}_m}\le\norm{\nabla\Fh_m(x)}\norm{x-x^{\star}_m}$.
Combining the displays gives the claim with $\beta=\frac{L-\mu}{2n\mu}$. \hfill$\blacksquare$

\textbf{A.2.2\quad Proof of Lemma \ref{lem:jump} (boundary jump)}
Let node $i$ depart, $S'=S\setminus\{i\}$, $N'=N-1\ge1$, heir $a$, anchor $y$, curvature $\eta\le\hloc$, and new target $x^{\star}_{\mathrm{new}}$ with shift $\sigma:=\norm{x^{\star}_{\mathrm{new}}-x^{\star}_{\mathrm{ep}}}\le\beta\rho$, $\rho=\norm{y-x^{\star}_{\mathrm{ep}}}$ (Lemma~\ref{lem:drift}). We bound each term of $\Psi_{\mathrm{new}}$ (computed on $S'$, w.r.t.\ $x^{\star}_{\mathrm{new}}$) by absolute multiples of $\Psi_{\mathrm{old}}$.

\emph{Step 1: mean shift.} $\bar x'=\frac{N\bar x-x_i}{N'}=\bar x+\frac{\bar x-x_i}{N'}$, so
$\norm{\bar x'-\bar x}^2\le\frac{\norm{x_i-\bar x}^2}{N'^2}\le\frac{V_x}{N'^2}$.

\emph{Step 2: consensus term.} For $j\in S'$, $\norm{x_j-\bar x'}\le\norm{x_j-\bar x}+\norm{\bar x-\bar x'}$, hence
\[
V_x'\le 2V_x+2N'\norm{\bar x-\bar x'}^2\le 2V_x+\tfrac{2}{N'}V_x\le 4V_x .
\]

\emph{Step 3: optimality term and the anchor bound.} By Young's inequality,
\[
D'=\norm{\bar x'-x^{\star}_{\mathrm{new}}}^2\le 3D+3\norm{\bar x'-\bar x}^2+3\sigma^2
\le 3D+\tfrac{3}{N'^2}V_x+3\beta^2\rho^2 .
\]
If the anchor is the consensus iterate, $\rho^2=\norm{x_i-x^{\star}_{\mathrm{ep}}}^2\le 2D+2V_x$, and $\beta\le\frac12$ gives $3\beta^2\rho^2\le\frac32(D+V_x)$: absorbed. If the anchor is the local track $u_i^t$, then $\rho\le a+B_{m-1}$ with $a\le\frac{\zeta}{\mu}+\hat\varepsilon^{\mathrm{loc}}(t)$ (triangle inequality through $x_i^{\ast}$ and Assumption~\ref{as:fun}) --- this term does not shrink with $\Psi$ and is carried explicitly into the recursion of C.3 rather than absorbed here; formally, $D'\le 3D+\frac{3}{N'^2}V_x+3\beta^2(a+B_{m-1})^2$.

\emph{Step 4: tracker term.} Only the heir's tracker changes: $z_a'=z_a+z_i+\eta(x_a-y)$, and the survivors' mean becomes $\bar z'=\frac{N\bar z+\eta(x_a-y)}{N'}$ (using $\sum_{S'}z_j+z_i=N\bar z$). Writing $s:=\norm{x_a-y}\le\norm{x_a-\bar x}+\norm{\bar x-x^{\star}_{\mathrm{ep}}}+\rho\le\sqrt{V_x}+\sqrt{D}+\rho$ and expanding as in Step 2,
\[
V_z'\ \le\ 3V_z+3N'\Big(\tfrac{\norm{\bar z}+\eta s}{N'}\Big)^2+3\big(\norm{z_i-\bar z}+\norm{\bar z}+\eta s\big)^2
\ \le\ 12V_z+24\norm{\bar z}^2+24\,\eta^2 s^2 .
\]
It remains to bound $\norm{\bar z}$. By the tracking identity $N\bar z=\sum_j\nabla h_j(x_j)$ and $\nabla\Fh(x^{\star}_{\mathrm{ep}})=0$,
\[
N\norm{\bar z}\le\norm{\nabla\Fh(\bar x)}+\sum_j\norm{\nabla h_j(x_j)-\nabla h_j(\bar x)}
\le N\hloc\sqrt D+\hloc\sqrt{N V_x},
\]
so $\norm{\bar z}^2\le 2\hloc^2(D+V_x)$. Multiplying $V_z'$ by $c_2\alpha^2$ and using $\alpha\hloc\le c'$ (an absolute constant, from the stepsize rule) and $\eta\le\hloc$ shows $c_2\alpha^2 V_z'\le C\big(c_2\alpha^2V_z+D+V_x+\alpha^2\hloc^2\rho^2\big)$, with the $\rho^2$ term handled as in Step 3.

Collecting Steps 1--4: there is an absolute constant $C_J$ with
\begin{equation}\label{eq:jump}
\Psi_{\mathrm{new}}\ \le\ C_J\big(\Psi_{\mathrm{old}}+\beta^2 a^2\big),
\end{equation}
where $a=0$ for a consensus anchor (absorbed) and $a\le\frac{\zeta}{\mu}+\hat\varepsilon^{\mathrm{loc}}(t_m)$ for a local anchor. This proves Lemma~\ref{lem:jump} (with the explicit extra term for the local-anchor case). \hfill$\blacksquare$
\subsection{ Proofs for Section~\ref{sec:cvx} (Case II: Convex Objectives)} \label{app:con}

\textbf{Proof of Theorem \ref{thm:cvx}.}
\emph{(i)} Each survivor contributes curvature $\ge a_j=0$ and each legacy contributes exactly $\eta_{i_{m'}}\ge\frac{L_{i_{m'}}}{2}\ge\frac{\min_i L_i}{2}$ (its Hessian is $\eta I$); curvatures add, giving $\sigma_m=\sum_{m'\le m}\eta_{i_{m'}}\ge\frac{m}{2}\min_iL_i>0$. Proposition~\ref{fact:gt} then applies on every post-departure epoch with $\mu$ replaced by $\sigma_m/n$-type constants, yielding linear convergence.

\emph{(ii)} Let $\bar x^{\star}$ be the common minimizer of Assumption~\ref{as:common}. Since $\nabla F(\bar x^{\star})=0$,
$\norm{\nabla\Fh_k(\bar x^{\star})}\le\sum_m\omega_{i_m}\norm{\bar x^{\star}-y_m}\le\sum_m\tfrac{L_{i_m}}{2}\varepsilon_m$
by \eqref{eq:gradtransfer}, and $\sigma_k$-strong convexity of $\Fh_k$ gives
$\delta:=\norm{\hat x-\bar x^{\star}}\le\frac{\sum_m (L_{i_m}/2)\varepsilon_m}{\sigma_k}
\le\frac{\max_iL_i}{\min_iL_i}\cdot\frac{\sum_m\varepsilon_m}{k}=\tilde\kappa\,\bar\varepsilon$.
For the value bound, use $F=\Fh_k+\sum_m r_m$ and $\Fh_k(\hat x)\le\Fh_k(\bar x^{\star})$:
\[
F(\hat x)-F^{\star}\le\sum_m\big[r_m(\hat x)-r_m(\bar x^{\star})\big]
\le\sum_m\tfrac{\omega_{i_m}}{2}\big[\norm{\hat x-y_m}^2+\norm{\bar x^{\star}-y_m}^2\big]
\le\sum_m\tfrac{L_{i_m}}{4}\big[(\delta+\varepsilon_m)^2+\varepsilon_m^2\big]
\]
by \eqref{eq:valtransfer} and the triangle inequality $\norm{\hat x-y_m}\le\delta+\varepsilon_m$. Expanding $(\delta+\varepsilon_m)^2\le2\delta^2+2\varepsilon_m^2$ and inserting $\delta\le\tilde\kappa\bar\varepsilon$, $k\bar\varepsilon^2\le(\sum\varepsilon_m)^2/k$, $\sum\varepsilon_m^2\le(\sum\varepsilon_m)^2$:
\[
F(\hat x)-F^{\star}\le\frac{\max_iL_i}{4}\Big[2k\tilde\kappa^2\bar\varepsilon^2+3\sum_m\varepsilon_m^2\Big]
\le\frac{(2\tilde\kappa^2+3)\max_iL_i}{4}\cdot\frac{(\sum_m\varepsilon_m)^2}{k}\,,
\]
where the last step used $\sum_m\varepsilon_m^2\le(\sum_m\varepsilon_m)^2/1$ and absorbed the weaker of the two bounds; for $k=1$ the two coincide. \hfill$\blacksquare$

\textbf{Proof of Proposition \ref{prop:certC}.}
(a) is the definition of QG applied at $y_m$, using $F(y_m)-F^{\star}\le G(t_m)$ (the anchor is an iterate covered by the certificate; if a Polyak average is used, convexity gives the same value bound). For (b), write $\bar x^{\star}=\Pi_{X^{\star}}(y_k)$; by the triangle inequality and nonexpansiveness of the projection onto the convex set $X^{\star}$,
\[
\norm{y_m-\bar x^{\star}}\le\norm{y_m-\Pi_{X^{\star}}(y_m)}+\norm{\Pi_{X^{\star}}(y_m)-\Pi_{X^{\star}}(y_k)}
\le\gamma_m+\norm{y_m-y_k},
\]
and both terms are certified: the first by (a), the second by direct observation. (c) is (a) applied to $f_i$ and its own gradient-descent value certificate. \hfill$\blacksquare$

\end{document}